\documentclass[reqno]{mcom-l}

\usepackage[utf8]{inputenc} 
\usepackage{graphicx} 
\usepackage{comment}
\usepackage{amsthm}
\usepackage{tikz}
\usetikzlibrary{calc}
\usetikzlibrary{patterns}
\tikzset{
  LabelStyle/.style = { rectangle, rounded corners, draw,
                        minimum width = 2em, fill = yellow!50,
                        text = red, font = \bfseries },
  VertexStyle/.append style = { inner sep=5pt,
                                font = \normalsize\bfseries},
  EdgeStyle/.append style = {->, bend left} }

\newtheorem{theorem}{Theorem}[section]
\newtheorem{lemma}[theorem]{Lemma}

\theoremstyle{definition}

\newcommand{\R}{\mathcal{R}}
\newcommand{\B}{\mathcal{B}}

\theoremstyle{remark}

\numberwithin{equation}{section}

\begin{document}

\title{Dots-and-Polygons}


\author{Jessica Dickson}
\address{}
\curraddr{}
\email{jmdickson@wsu.edu}
\thanks{}

\author{Rachel Perrier}
\address{}
\curraddr{}
\email{rachel.perrier@wsu.edu}
\thanks{}


\date{}

\dedicatory{}

\begin{abstract}
Dots-and-Boxes is a popular children's game whose winning strategies have been studied by Berlekamp, Conway, Guy, and others. In this article we consider two variations, Dots-and-Triangles and Dots-and-Polygons, both of which utilize the same lattice game board structure as Dots-and-Boxes. The nature of these variations along with this lattice structure lends itself to applying Pick's theorem to calculate claimed area. Several strategies similar to those studied in Dots-and-Boxes are used to analyze these new variations.
\end{abstract}

\maketitle

\section{Introduction}
Dots-and-Boxes is a children's game in which dots are set up in a rectangular grid.  Two players take turns drawing horizontal and vertical lines between these dots.  If a player completes a $1\times 1$ box they get to claim this area by placing their initials inside it or shading it with their color.  This player also gets to draw an extra line after claiming a box.  The game ends once the entire game board is claimed, and the winner is the player who has claimed the most boxes.

The game of Dots-and-Boxes has been studied extensively by Elwyn Berlekamp, John Conway, and Richard Guy \cite{Berlekamp,Win}.  In their books on the topic, they reveal strategies players can employ to give them a better chance to win.  One main strategy is the double-dealing move.  In this move, rather than claiming two boxes, the first player makes a move that allows the second player to claim both in a single move.  While the first player may lose these two boxes, this double-dealing move also forces the second player to move first in an unclaimed region of the game board, often with the purpose to open up a longer chain of boxes for the first player to claim.  The move where the second player claims the two boxes in a single move is called a doublecrossed move.  Figure \ref{IntroDC} demonstrates these types of moves.
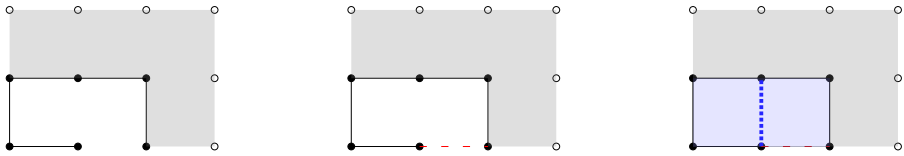
\begin{figure}[htb]
\centering
\resizebox{0.95\textwidth}{!}{%
\begin{tikzpicture}
    \foreach \Point in {(0,2),(1,2),(2,2),(3,2),(3,1),(3,0)}{
        \node[draw,circle,inner sep=1pt] at \Point {};
    }
    \foreach \x in {0,1,2}{
      \foreach \y in {0,1}{
        \node[draw,circle,inner sep=1pt,fill] at (\x,\y) {};
      }
    }
    \draw (1,0) -- (0,0) -- (0,1) -- (2,1) -- (2,0);
    \fill[gray,nearly transparent]  (2,0) -- (3,0) -- (3,2) -- (0,2) -- (0,1) -- (2,1) -- cycle;
    
    \foreach \Point in {(5,2),(6,2),(7,2),(8,2),(8,1),(8,0)}{
        \node[draw,circle,inner sep=1pt] at \Point {};
    }
    \foreach \x in {5,6,7}{
      \foreach \y in {0,1}{
        \node[draw,circle,inner sep=1pt,fill] at (\x,\y) {};
      }
    }
    \draw (6,0) -- (5,0) -- (5,1) -- (7,1) -- (7,0);
    \draw[thin, red, loosely dashed] (6,0) -- (7,0);
    \fill[gray,nearly transparent]  (7,0) -- (8,0) -- (8,2) -- (5,2) -- (5,1) -- (7,1) -- cycle;
    
    \foreach \Point in {(10,2),(11,2),(12,2),(13,2),(13,1),(13,0)}{
        \node[draw,circle,inner sep=1pt] at \Point {};
    }
    \foreach \x in {10,11,12}{
      \foreach \y in {0,1}{
        \node[draw,circle,inner sep=1pt,fill] at (\x,\y) {};
      }
    }
    \draw (11,0) -- (10,0) -- (10,1) -- (12,1) -- (12,0);
    \draw[ultra thick, blue, densely dotted] (11,0) -- (11,1);
    \draw[thin, red, loosely dashed] (11,0) -- (12,0);
    \fill[gray,nearly transparent]  (12,0) -- (13,0) -- (13,2) -- (10,2) -- (10,1) -- (12,1) -- cycle;
    \draw[fill=blue!40,nearly transparent] (10,0) rectangle (12,1);
\end{tikzpicture} 
}
\caption{The left picture shows the board on which the players are about to move.  The middle picture shows the first player's double-dealing move (in red dashes) while the right picture demonstrates the second player performing a doublecrossed move and claiming area (using blue dotted lines).}
\label{IntroDC}
\end{figure}

 Berlekamp, Conway, and Guy also give other strategies for controlling the form and final outcome of a game. These strategies rely on the number of dots on the game board, the number of long chains, and the number of doublecrossed moves.   They also mention a few variations of Dots-and-Boxes, such as playing on other board shapes and playing a variation called Strings-and-Coins on the dual of a Dots-and-Boxes game board.

In this paper we explore a new variation of Dots-and-Boxes, one where closing any polygonal-shaped region allows a player to claim it.  In this situation, a winner could be considered the player who captured the most area on the grid, similar to the player who claims the most boxes in Dots-and-Boxes. In the latter game, the unit nature of the boxes allows area to be easily calculated. The question then becomes, is there a simple way to calculate area in the former game?

Pick's Theorem allows easy calculation of the area of a polygon on an integer lattice, provided that all corner points are integer \cite{Gaskell}.  Let $B$ be the number of boundary points, that is, the lattice points that lie on the boundary of the polygon, and $I$ be the number of interior lattice points of the polygon.  Then, Pick's Theorem  tells us that the area of the closed region is given by
$$I+\frac{B}{2}-1.$$
For instance, consider Figure \ref{Pick}.  The polygon has $12$ boundary points and $2$ interior points, which makes the area of this polygon $2+\frac{12}{2}-1=7$.  This result can be verified by cutting the polygon up into geometrical shapes with known area formulas.

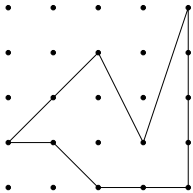
\begin{figure}[htb]
\centering
\resizebox{0.2\textwidth}{!}{%
\begin{tikzpicture}
    \foreach \x in {0,1,2,3,4}{
      \foreach \y in {0,1,2,3,4}{
        \node[draw,circle,inner sep=1pt,fill] at (\x,\y) {};
      }
    }
    \draw (0,1) -- (2,3) -- (3,1) -- (4,4) -- (4,0) -- (2,0) -- (1,1) -- cycle;

\end{tikzpicture} 
}%
\caption{Example of a closed region on the integer lattice whose area can be calculated using Pick's Theorem.}
\label{Pick}
\end{figure}

We will consider two variations of Dots-and-Boxes, both with the goal to claim the most area. In Dots-and-Triangles players claim area on the lattice game board by completing triangles with three boundary points and no interior points. Pick's theorem tells us that these claimed triangles have area $\frac{1}{2}$.  Since all these triangles have the same area, the player who claims the most will be the winner. Later, we consider the game Dots-and-Polygons, where players gain area by closing polygonal regions with no interior line segments.

\section{Dots-and-Triangles}

The game of Dots-and-Triangles is played on an finite integer lattice structure where players alternate turns. We will assume that these games are always played on finite rectangular lattice structures. The rules are as follows: (1) Each turn consists of connecting any two distinct points with a straight line that does not intersect another lattice point nor another line in the process; (2) if a player closes a triangular region with three boundary points and without any interior points, that player claims the associated area and immediately makes another move; (3) two players Robert and Betty play alternately, with the assumption Robert moves first; (4) the game ends once all the area on the game board is claimed.   For the remainder of this paper, we will assume that Robert (denoted by $\R$ with the use of red dashed lines) is always the first to move while Betty (denoted by $\B$ with blue dotted lines) is second. 

Unlike the version in \cite{Win} where the triangular cells are predetermined, the Dots-and-Triangles version we propose is more flexible.  Our proposal resembles a lattice structure version of the Monochromatic Complete Triangulation game discussed in \cite{Aichholzer} and the sankaku-tori game discussed in \cite{10.1007/978-3-319-07890-8_20}.  Our version plays out differently than these games, though, as neither allows collinear points.  Additionally, sankaku-tori lacks the free turn after the completion of a triangle. However, like the Monochromatic Complete Triangulation game or sankaku-tori, different plays on the same game board may result in different structures, as is demonstrated in Figure \ref{Moves}. Consequently, this rules out a helpful strategy used in analyzing Dots-and-Boxes---the dual graph, Strings-and-Coins. Since the dual structure of a Dots-and-Triangles game board is not evident until players have carved out the triangular cells, an analysis of this dual nature will not provide much insight until the game is nearly over.

\begin{figure}[htb]
\centering
\resizebox{0.5\textwidth}{!}{%
\begin{tikzpicture}
    \foreach \x in {0,1,2}{
      \foreach \y in {0,1,2}{
        \node[draw,circle,inner sep=1pt,fill] at (\x,\y) {};
      }
    }
    \draw (0,0) -- (0,2) -- (1,0) -- cycle;
    \draw (0,1) -- (1,0);
    \draw (1,0) -- (2,2) -- (2,0) -- cycle;
    \draw (1,0) -- (2,1);
    \draw (0,2) -- (1,1) -- (2,2) -- cycle;
    \draw (1,0) -- (1,2);
    
    \foreach \x in {5,6,7}{
      \foreach \y in {0,1,2}{
        \node[draw,circle,inner sep=1pt,fill] at (\x,\y) {};
      }
    }
    \draw (5,0) -- (5,2) -- (7,2) -- (7,0) -- cycle;
    \draw (5,0) -- (7,2);
    \draw (7,0) -- (5,2);
    \draw (6,0) -- (6,2);
    \draw (5,1) -- (7,1);
\end{tikzpicture} 
}%
\caption{Example of two possible final game board formations for a $3\times 3$ game of Dots-and-Triangles.}
\label{Moves}
\end{figure}
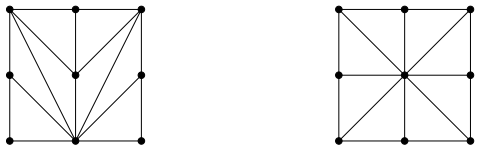

\subsection{Double-Dealings and Doublecrosses}

As was mentioned in the introduction, one of the main winning strategies employed in Dots-and-Boxes is forcing doublecrossed moves. These often occur at the end of a \textit{long chain}, a series of three or more boxes than can be completed in a single turn. Such a strategy allows a player to maintain control of the game. It is shown in \cite{Berlekamp,Win} that in the game of Dots-and-Boxes the number of turns equals the number of starting dots on the game board plus the number of doublecrosses. This result provides a clear strategy for the players. The first player wants to make (turns) = (dots) + (doublecrossed moves) odd; that is, (dots) + (long chains) even. Meanwhile the second player tries to make (turns) = (dots) + (doublecrossed) moves even; that is, (dots) + (long chains) odd.  It is interesting to note that, despite Dots-and-Triangles generating twice as many closed shapes as Dots-and-Boxes, the same strategies still hold.

\begin{theorem}
The number of turns equals the number of dots on the game board plus the number of doublecrosses.
\label{TrianglesDCThm}
\end{theorem}

\begin{proof}
 This proof is a combination of those from \cite{Win,VardiTurns} using triangular cells rather than square ones and relies on a result from graph theory.
 
 First, suppose we are playing a game of Dots-and-Triangles without any doublecrosses. Let $D$ be the number of dots, $T$ the number of turns to draw $L$ line segments, and $P$ be the number of triangles we finish the game with. Now, every line segment that is placed, except the last, either forms exactly one triangle or ends the turn (but not both). However, the last line segment placed forms a triangle and marks the end of a turn and thus is counted in both $P$ and $T$. A $1$ is subtracted to rectify this double counting. This gives us the formula:
 \begin{align}
     L&=P+T-1.
 \end{align}
 But Euler's formula, which applies to finite, connected, planar graphs without edge intersections, can also be applied in a modified way to Dots-and-Triangles. That is, since we have $P+1$ faces in Euler's formula (rather than $P$ as we must count the outer face), we have
 \begin{align}
    D-L+(P+1) &= 2 \nonumber
 \end{align}
 or
 \begin{align}
    P + D - 1 &= L.
 \end{align}
So in a game without doublecrosses, Equations $2.1$ and $2.2$ give us that the number of turns equals the number of dots ($D=T$).

Now let's consider a game with doublecrosses, and let $C$ be the number of doublecrosses. Each doublecross will create two triangles rather than one so we can adjust for this double counting by subtracting $C$ from $P$. Altering Equation $2.1$ accordingly gives
\begin{align}
    L&=P-C+T-1
\end{align}
and, using Equation $2.2$, which still holds, combined with Equation $2.3$, we get
 $$T=D+C.$$
That is, the number of turns equals the number of dots plus the number of doublecrosses.
\end{proof}

 \subsection{Nested Diamonds}
For the purposes of this game, a \textit{diamond} is a quadrilateral whose boundaries have slopes of $1$ or $-1$. An \textit{$n$-nested diamond}, is a set of $n$ diamonds ($n\geq 1$) that lie within one another and have no other interior points between the layers of diamonds, with the exception of the innermost lattice point. Note that a $1$-nested diamond is just a single diamond with one interior point. Figure \ref{ND} presents an example of a 3-nested diamond.

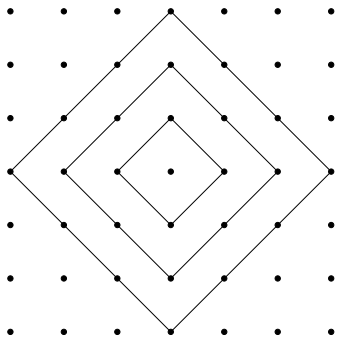
\begin{figure}[htb]
\centering
\resizebox{0.35\textwidth}{!}{%
\begin{tikzpicture}
    \foreach \x in {0,1,...,6}{
      \foreach \y in {0,1,...,6}{
        \node[draw,circle,inner sep=1pt,fill] at (\x,\y) {};
      }
    }
    \draw (0,3) -- (3,0) -- (6,3) -- (3,6) -- cycle;
    \draw (3,1) -- (1,3) -- (3,5) -- (5,3) -- cycle;
    \draw (3,2) -- (2,3) -- (3,4) -- (4,3) -- cycle;
\end{tikzpicture} 
}
\caption{Example of a 3-nested diamond.}
\label{ND}
\end{figure}

Now let's examine winning area on an $n$-nested diamond. If Robert is using the strategy of an ordinary child, as described in \cite{Berlekamp,Win}, he will move in the innermost nested diamond in order to give up the least area. If Betty were also playing as an ordinary child, she would then claim all of that area, and move in the next smallest nested diamond. In this manner, Robert wins if there is an even number of nested diamonds and Betty wins if there is an odd number.

However, if Betty utilizes a double-dealing strategy, she can always win the majority of the area of an $n$-nested diamond. Consider the subgame of playing in the 2-nested diamond shown in Figure \ref{NDPlay}. In this instance, Betty plays using a double-dealing strategy. She will not claim the inner area; instead, Betty will finish off the inner straight line begun by Robert.  Then, Robert would claim the inner region, and be the first to move in the next layer after which Betty would claim all the area of the second layer.  For a $3$-nested diamond, Betty could double-deal again in the second layer, as illustrated in Figure \ref{NDPlay2}, and then claim the entirety of the third layer. In fact, if Betty uses this strategy for any $n$-nested diamond where $n>1$, she will always claim more area than Robert.

\begin{figure}[htb]
\centering
\resizebox{0.95\textwidth}{!}{%
\begin{tikzpicture}
    \foreach \x in {0,1,...,4}{
      \foreach \y in {0,1,...,4}{
        \node[draw,circle,inner sep=1pt,fill] at (\x,\y) {};
      }
    }
    \draw (0,2) -- (2,0) -- (4,2) -- (2,4) -- cycle;
    \draw (2,1) -- (1,2) -- (2,3) -- (3,2) -- cycle;
    \draw[thin, red, loosely dashed] (1,2) -- (2,2);
    
    \foreach \x in {6,7,...,10}{
      \foreach \y in {0,1,...,4}{
        \node[draw,circle,inner sep=1pt,fill] at (\x,\y) {};
      }
    }
    \draw (6,2) -- (8,0) -- (10,2) -- (8,4) -- cycle;
    \draw (8,1) -- (7,2) -- (8,3) -- (9,2) -- cycle;
    \draw[thin, red, loosely dashed] (7,2) -- (8,2);
    \draw[ultra thick, blue, densely dotted] (8,2) -- (9,2);
    
    \foreach \x in {12,13,...,16}{
      \foreach \y in {0,1,...,4}{
        \node[draw,circle,inner sep=1pt,fill] at (\x,\y) {};
      }
    }
    \draw (12,2) -- (14,0) -- (16,2) -- (14,4) -- cycle;
    \draw (14,1) -- (13,2) -- (14,3) -- (15,2) -- cycle;
    \draw[thin, red, loosely dashed] (13,2) -- (14,2);
    \draw[ultra thick, blue, densely dotted] (14,2) -- (15,2);
    \draw[thin, red, loosely dashed] (14,1) -- (14,4);
    \draw[fill=red!40,nearly transparent]  (14,1) -- (13,2) -- (14,3) -- (15,2) -- cycle;
\end{tikzpicture} 
}
\caption{Example of $\B$'s double-dealing strategy.}
\label{NDPlay}
\end{figure}
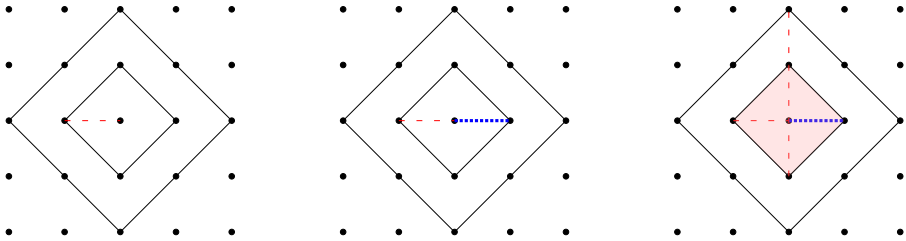

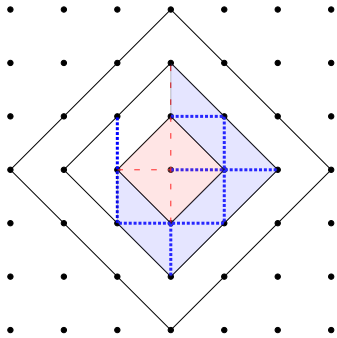
\begin{figure}[htb]
\centering
\resizebox{0.35\textwidth}{!}{%
\begin{tikzpicture}
    \foreach \x in {0,1,...,6}{
      \foreach \y in {0,1,...,6}{
        \node[draw,circle,inner sep=1pt,fill] at (\x,\y) {};
      }
    }
    \draw (0,3) -- (3,0) -- (6,3) -- (3,6) -- cycle;
    \draw (3,1) -- (1,3) -- (3,5) -- (5,3) -- cycle;
    \draw (3,2) -- (2,3) -- (3,4) -- (4,3) -- cycle;
    \draw[thin, red, loosely dashed] (2,3) -- (3,3);
    \draw[ultra thick, blue, densely dotted] (3,3) -- (4,3);
    \draw[thin, red, loosely dashed] (3,2) -- (3,5);
    \draw[fill=red!40,nearly transparent]  (3,2) -- (2,3) -- (3,4) -- (4,3) -- cycle;
    \draw[ultra thick, blue, densely dotted] (3,4) -- (4,4) -- (4,3) -- (5,3);
    \draw[ultra thick, blue, densely dotted] (4,3) -- (4,2) -- (3,2) -- (3,1);
    \draw[ultra thick, blue, densely dotted] (3,2) -- (2,2) -- (2,4);
    \draw[fill=blue!40,nearly transparent] (2,3) -- (2,2) -- (3,1) -- (5,3) -- (3,5) -- (3,4) -- (4,3) -- (3,2) -- cycle;
\end{tikzpicture} 
}
\caption{Example of $\B$ continuing her double-dealing strategy from Figure \ref{NDPlay} in the second layer of a 3-nested diamond.}
\label{NDPlay2}
\end{figure}

\begin{theorem}
In an $n$-nested diamond, the second player to move can always win, optimally claiming an area of $2n^2-2n+2$.
\end{theorem}

\begin{proof}
Notice that the nested diamond structure can be broken into four congruent right triangles, each with two legs of length $n$.  So, the total area composed in this shape is $4 (\frac{1}{2}n^2)=2n^2$.

In a $1$-nested diamond, Robert moves first, leaving Betty to win by claiming all 4 triangles for a total area of $2=2(1^2)-2(1)+2$. For $n>1$, until the last layer of the $n$-nested diamond is played in, Betty, using the strategy described previously, will perform a double-dealing move, forcing Robert to always move first in any layer. This means Robert will only claim 4 triangles (or an area of 2) in each layer, excluding the last one where Betty claims all of the area. Thus Robert claims an area of $2(n-1)$ out of a total area of $2n^2$, giving Betty a total area of $2n^2-2(n-1)=2n^2-2n+2$. Since $2n^2-2n+2 > \frac{1}{2} (2n^2)=n^2$, Betty will win the majority of the area.
\end{proof}

Note that this theorem only decides who wins the majority of the area in the $n$-nested diamond. On a rectangular grid structure, if an $n$-nested diamond takes up the entire center of a $(2n+1)\times (2n+1)$ game board (which has total claimable area $4n^2$), then Betty has taken less than half the area possible.  This means that the play that happens in the four corners would determine the winner of the game. 

\subsection{Reduced and Extremely Reduced Shapes}
 
We use \textit{shape} as a general term to describe a closed polygon that may have interior line segments and has not been fully claimed. In a \textit{reduced} shape, any two boundary points that can be validly connected via the interior of the shape would divide it into two subshapes, each with at least one interior point.  That is, in a reduced shape no player can immediately claim area in a single move.   An \textit{extremely reduced} shape is a reduced shape that cannot be subdivided in a single move into two reduced shapes. Note that this means the boundary points of an extremely reduced shape cannot be validly connected in one move.
 
\begin{theorem}
\label{nointerior}
Any closed shape with no interior points is not reduced. Moreover, such a shape can be completely claimed by a player in a single turn.
\end{theorem}
\begin{proof}
We propose an algorithm to show how all area in a closed shape with no interior points can be claimed by a player in a single turn. Per \cite{Gaskell}, any shape can be decomposed into triangles of area $\frac{1}{2}$. Consider any such decomposition. Figure \ref{PicNoInterior} illustrates a possible decomposition and demonstrates the following two steps on an example shape. 

\textit{Step 1:} Since the boundary of the shape would need to be included in this decomposition, begin by locating the triangles that are formed from two boundary line segments and one interior line segment. Since the boundary line segments have already been drawn, drawing these interior line segments will grant a free turn as each will complete a triangle of area $\frac{1}{2}$. If these moves claim all the area, we have our result. Otherwise, proceed to step 2.

\textit{Step 2:} We are now left with a subshape with no interior points and a decomposition originating from the starting decomposition. Perform step 1 again on this subshape. 

Since we are considering shapes on a finite board, there are only a finite number of moves possible.  Because of this, the proposed algorithm will terminate in a finite number of iterations and will result in all area claimed in a single turn.  This also implies that this closed shape was not reduced, as area was able to be claimed.
\end{proof}

\begin{figure}[htb]
\centering
\resizebox{0.95\textwidth}{!}{%
\begin{tikzpicture}
    \foreach \x in {0,1,...,3}{
      \foreach \y in {0,1,...,3}{
        \node[draw,circle,inner sep=1pt,fill] at (\x,\y) {};
      }
    }
    \draw[thick] (1,0) -- (1,1) -- (2,1) -- (0,3) -- (2,2) -- (2,3) -- (3,0) -- cycle;
    \draw[gray] (1,0) -- (2,1);
    \draw[gray] (2,0) -- (2,1);
    \draw[gray] (3,0) -- (2,1);
    \draw[gray] (3,0) -- (2,2);
    \draw[gray] (2,2) -- (2,1);
    \draw[gray] (1,2) -- (2,2);
    
    \foreach \x in {5,6,...,8}{
      \foreach \y in {0,1,...,3}{
        \node[draw,circle,inner sep=1pt,fill] at (\x,\y) {};
      }
    }
    \draw[thick] (6,0) -- (6,1) -- (7,1) -- (5,3) -- (7,2) -- (7,3) -- (8,0) -- cycle;
    \draw[magenta,thick] (6,0) -- (7,1);
    \draw[gray] (7,0) -- (7,1);
    \draw[gray] (8,0) -- (7,1);
    \draw[magenta,thick] (8,0) -- (7,2);
    \draw[gray] (7,2) -- (7,1);
    \draw[magenta,thick] (6,2) -- (7,2);
    \fill[magenta!40,nearly transparent] (6,0) -- (6,1) -- (7,1) -- cycle;
    \fill[magenta!40,nearly transparent] (5,3) -- (6,2) -- (7,2) -- cycle;
    \fill[magenta!40,nearly transparent] (8,0) -- (7,3) -- (7,2) -- cycle;
    
    \foreach \x in {10,11,...,13}{
      \foreach \y in {0,1,...,3}{
        \node[draw,circle,inner sep=1pt,fill] at (\x,\y) {};
      }
    }
    \draw[thick] (11,0) -- (12,1) -- (11,2) -- (12,2) -- (13,0) -- cycle;
    \draw[gray] (11,0) -- (11,1) -- (12,1);
    \draw[gray] (11,2) -- (10,3) -- (12,2) -- (12,3) -- (13,0);
    \draw[gray] (12,0) -- (12,1);
    \draw[gray] (13,0) -- (12,1);
    \draw[gray] (12,2) -- (12,1);
    \fill[gray!40,nearly transparent] (11,0) -- (11,1) -- (12,1) -- cycle;
    \fill[gray!40,nearly transparent] (10,3) -- (11,2) -- (12,2) -- cycle;
    \fill[gray!40,nearly transparent] (13,0) -- (12,3) -- (12,2) -- cycle;

\end{tikzpicture} 
}
\caption{Example of the algorithmic approach for Theorem \ref{nointerior}. The left picture shows a possible decomposition of an example starting shape into triangles of area $\frac{1}{2}$.  The middle picture shows the area claimed by a player in step 1.  The right picture shows the resulting smaller subshape and its decomposition to use for step 2.}
\label{PicNoInterior}
\end{figure}
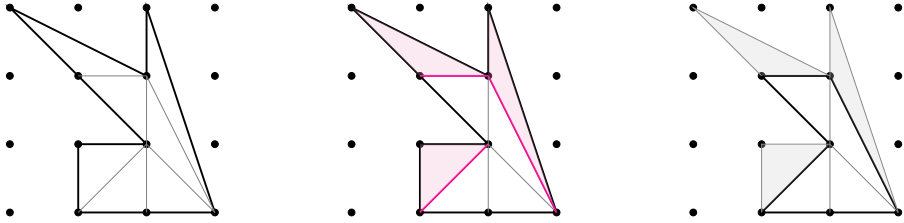

It is important to note that although the above theorem shows that all area of a closed shape with no interior points can be claimed in a single turn, it may not be in a player's best interest to do so.  A doublecross situation may prove to be more valuable, in a similar manner discussed in the nested diamond formation.
 
We can use Theorem \ref{nointerior} and the following Lemma \ref{GCDTheorem} to make a claim about a property of extremely reduced shapes. In Theorem \ref{ConvexERS}, we refer to \textit{linear points}, which are lattice points on the boundary of a polygon that are not corner points.

\begin{lemma}
If there are $k+1$ lattice points on a line segment, including the lattice endpoints $(a,b)$ and $(c,d)$, then $k=\gcd(c-a,d-b).$

\label{GCDTheorem}
\end{lemma}
\begin{proof}
Consider a line segment from point $(a,b)$ to $(c,d)$ with $k+1$ lattice points, including endpoints, and with slope $\frac{r}{s}$ where $\gcd(r,s)=1$. Using Theorem 3.8 from Apostol in \cite{GCDCount}, we know there are no integer lattice points between $(a+(n-1)s, b+(n-1)r)$ and $(a+ns, b+nr)$ for $1\leq n \leq k$. Thus we can label each point in this manner, specifically $(c,d)$ as $(a+ks, b+kr)$, and get
\begin{align*}
    \gcd(c-a, d-b) &= \gcd(a+ks-a, b+kr-b)\\
        &= \gcd(ks,rk)\\
        &= k.
\end{align*}  
Hence, we have our result.
\end{proof}

\begin{theorem}
There are no convex extremely reduced shapes without interior line segments that have five boundary points.
\label{ConvexERS}
\end{theorem}

\begin{proof}
We establish this proof by contradiction. A convex extremely reduced shape with exactly five boundary points falls into three cases: three corner points and two linear points, four corner points and one linear point, or five corner points.  By Theorem \ref{nointerior}, we know that the shape in question has at least one interior point.
 
We will first consider the case with four corner points and one linear point. Assume that the shape is extremely reduced and thus that any line segment between any two non-adjacent boundary points would have to cross an interior lattice point. Since the shape is extremely reduced, we can apply Lemma \ref{GCDTheorem} to two boundary points $(a,b)$ and $(c,d)$. If these points are adjacent, then $\gcd(c-a,d-b)=1$ and if these points are non-adjacent, then  $\gcd(c-a,d-b)>1$.

Assume without loss of generality that $(0,0), (n,m),\text{ and } (2n,2m)$ are boundary points, where $\gcd(m,n) = 1$, and that the other two boundary points are $x=(x_1, x_2)$ and $y=(y_1, y_2)$.  (We can always use the rigid transformations of translation and rotation to move the shape to these positions.)  Note that $(n,m)$ is the linear point, and the rest are corner points.  Further, assume that there are line segments connecting $x$ to $(0,0)$, $y$ to $(2n,2m)$, and $x$ to $y$. Then, using Lemma \ref{GCDTheorem}, we have the requirements:
\begin{align*}
    \gcd(y_1 - 0, y_2 - 0) &= d_1 > 1, \\
    \gcd(y_1 - n, y_2 - m) &= d_2 > 1, \\
    \gcd(x_1 - n, x_2 - m) &= d_3 > 1, \\
    \gcd(x_1 - 2n, x_2 - 2m) &= d_4 > 1, \\
    \gcd(x_1 - 0, x_2 - 0) &= 1, \\
    \gcd(y_1 - 2n, y_2 - 2m) &= 1, \\
    \gcd(x_1 - y_1, x_2 - y_2) &= 1. 
\end{align*}
Note that $d_i$ for $i=1,2,3,4$ above are integers.  By B\'ezout's identity (see \textsection 1.2 in \cite{Bezout} for details), we can find integers $a_i, b_i, c_i, e_i, f_i, g_i, h_i$, for $i=1,2$, such that
\begin{align*}
    a_1 y_1 + a_2 y_2 &= d_1, \\
    b_1 (y_1 - n) + b_2 (y_2 - m) &= d_2, \\
    c_1 (x_1 - n) + c_2 (x_2 - m) &= d_3, \\
    e_1 (x_1 - 2n) + e_2 (x_2 - 2m) &= d_4, \\
    f_1 x_1 + f_2 x_2 &= 1, \\
    g_1 (y_1 - 2n) + g_2 (y_2 - 2m) &= 1,\\
    h_1 (x_1 - y_1) + h_2 (x_2 - y_2) &= 1.
\end{align*}
Rearranging these equations and putting the coefficients into an augmented matrix, we obtain
$$
  \left[ 
    \begin{matrix}
      0 & 0 & a_1 & a_2 & d_1 \\
      0 & 0 & b_1 & b_2 & b_1 n + b_2 m + d_2 \\
      c_1 & c_2 & 0 & 0 & c_1 n + c_2 m + d_3 \\
      e_1 & e_2 & 0 & 0 & 2e_1 n + 2e_2 m + d_4\\
      f_1 & f_2 & 0 & 0 & 1 \\
      0 & 0 & g_1 & g_2 & 2g_1 n + 2g_2 m+  1 \\
      h_1 & h_2 & -h_1 & -h_2 & 1
    \end{matrix}
    \right]
$$
which row reduces to
$$
  \left[ 
    \begin{matrix}
      1 & 0 & 0 & 0 & 0 \\
      0 & 1 & 0 & 0 & 0 \\
      0 & 0 & 1 & 0 & 0 \\
      0 & 0 & 0 & 1 & 0 \\
      0 & 0 & 0 & 0 & 1 \\
      0 & 0 & 0 & 0 & 0 \\
      0 & 0 & 0 & 0 & 0
    \end{matrix}
    \right].
$$
However, this tells us that the system has no solution, which contradicts our assumption that our shape is extremely reduced.  Thus, this case does not yield an extremely reduced shape.
 
The case with five corner points has more variables, as the points can be labeled as $(0,0), w=(w_1, w_2), x=(x_1, x_2), y=(y_1, y_2)$ and $z=(z_1, z_2)$, where the adjacency is given by the ordering. However, by using Lemma \ref{GCDTheorem}, it can be shown that the system of equations given by the Euclidean algorithm is again inconsistent. The logistics are similar to the first case and are omitted for brevity's sake.

The final case of three corner points and two collinear points splits into two subcases dependent on whether those collinear points are adjacent. If they are adjacent, $(0,0), (n,m), (2n, 2m), (3n, 3m)$ and $x=(x_1, x_2)$ are these boundary points, where the $\gcd(n,m)=1$ and $x$ is connected to both $(0,0)$ and $(3n, 3m)$. A similar process to the previous two cases will show this subcase is impossible.  If the collinear points are not adjacent, label them $(0,0), (n_1, n_2), (2n_1, 2n_2), (m_1, m_2), (2m_1, 2m_2)$ where $(n_1, n_2)$ and $(m_1, m_2)$ are the collinear points and $\gcd(n_1, n_2)=\gcd(m_1,m_2)=1$. But then ${\gcd(2m_1-2n_1, 2m_1-2n_2) = 1}$ which is false, telling us this subcase is impossible. 
 \end{proof}

This theorem says that a player can always find two boundary points to connect in a convex shape with five boundary points and no interior line segments. Generically speaking, a player is often able to claim triangles from the boundary of a convex shape until it is reduced.  Alternatively (or additionally), if the player cannot gain area, they can split the shape into two smaller shapes without ceding area to their opponent.  Figure \ref{ShaveOff} gives an example of claiming triangles to get a reduced shape and splitting the resulting reduced shape into two extremely reduced subshapes.  It is interesting to note that Theorem \ref{ConvexERS} does not extend to the six boundary point case, as can be observed in Figure \ref{6Convex} or in \cite{373673}.

\begin{figure}[htb]
\centering
\resizebox{1\textwidth}{!}{%
\begin{tikzpicture}
    \foreach \x in {0,1,...,4}{
      \foreach \y in {0,1,...,4}{
        \node[draw,circle,inner sep=1pt,fill] at (\x,\y) {};
      }
    }
    \draw (0,1) -- (0,2) -- (2,4) -- (4,3) -- (3,0) -- (1,0) -- cycle;
    \foreach \x in {6,7,...,10}{
      \foreach \y in {0,1,...,4}{
        \node[draw,circle,inner sep=1pt,fill] at (\x,\y) {};
      }
    }
    \draw (6,1) -- (6,2) -- (8,4) -- (10,3) -- (9,0) -- (7,0) -- cycle;
    \draw[magenta, ultra thick] (8,0) -- (6,1) -- (9,0);
    \fill[magenta!40, nearly transparent] (7,0) -- (6,1) -- (9,0) -- cycle;
    \draw[magenta, ultra thick] (7,3) -- (6,1) -- (8,4);
    \fill[magenta!40, nearly transparent] (6,1) -- (8,4) -- (6,2) -- cycle;
    \foreach \x in {12,13,...,16}{
      \foreach \y in {0,1,...,4}{
        \node[draw,circle,inner sep=1pt,fill] at (\x,\y) {};
      }
    }
    \draw (12,1) -- (14,4) -- (16,3) -- (15,0) -- cycle;
    \draw[magenta, ultra thick] (14,4) -- (15,0);
    \draw[gray] (12,1) -- (12,2) -- (14,4) -- (16,3) -- (15,0) -- (13,0) -- cycle;
    \draw[gray] (13,3) -- (12,1) -- (14,0);
    \fill[gray!40, nearly transparent] (12,1) -- (12,2) -- (14,4) -- cycle;
    \fill[gray!40, nearly transparent] (12,1) -- (15,0) -- (13,0) -- cycle;

\end{tikzpicture} 
}%
\caption{The starting convex shape with no interior line segments is shown on the left. In the center, triangles have been claimed until the shape is reduced. On the right, a final move is made to produce two extremely reduced shapes.}
\label{ShaveOff}
\end{figure}
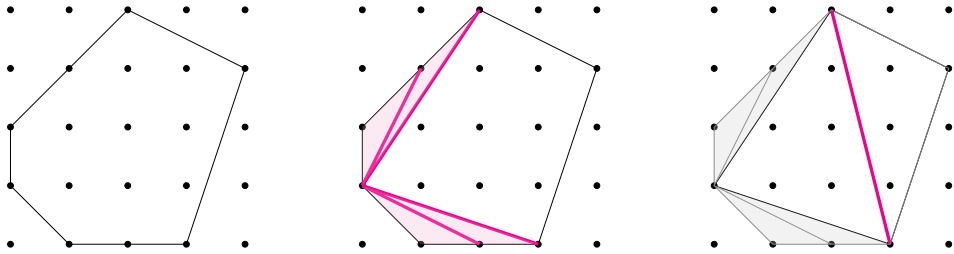

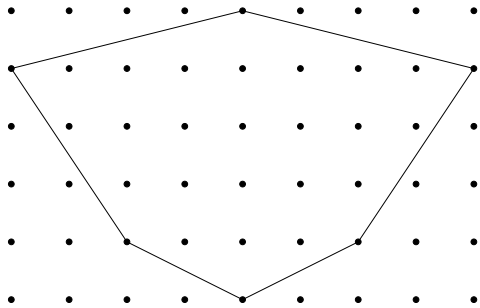
\begin{figure}[htb]
\centering
\resizebox{0.5\textwidth}{!}{%
\begin{tikzpicture}
    \foreach \x in {0,1,...,8}{
      \foreach \y in {0,1,...,5}{
        \node[draw,circle,inner sep=1pt,fill] at (\x,\y) {};
      }
    }
    \draw (4,0) -- (2,1) -- (0,4) -- (4,5) -- (8,4) -- (6,1) -- cycle;

\end{tikzpicture} 
}%
\caption{An extremely reduced convex shape with six boundary points.}
\label{6Convex}
\end{figure}

In the previous few theorems we exclusively considered shapes with no interior line segments. Let us now switch our focus to what we shall define as an \textit{eye},  a closed polygon in which all interior points are path connected, but none of the interior points are connected to boundary points. We use the term \textit{iris} to describe the path connected interior points and line segments. For instance, in Figure \ref{E} we see two examples of eyes.

\begin{figure}[htb]
\centering
\resizebox{0.65\textwidth}{!}{%
\begin{tikzpicture}
    \foreach \x in {-1,0,...,3}{
      \foreach \y in {0,1,...,4}{
        \node[draw,circle,inner sep=1pt,fill] at (\x,\y) {};
      }
    }
    \draw (0,1) -- (1,0) -- (2,1) -- (2,3) -- (1,4) -- (0,3) -- cycle;
    \draw (1,1) -- (1,3);
    
    \foreach \x in {5,6,...,9}{
      \foreach \y in {0,1,...,4}{
        \node[draw,circle,inner sep=1pt,fill] at (\x,\y) {};
      }
    }
    \draw (6,0) -- (9,2) -- (7,3) -- (8,4) -- (5,2) -- cycle;
    \draw (6,1) -- (7,1) -- (8,2) -- cycle;
    \fill[gray!40, nearly transparent] (6,1) -- (7,1) -- (8,2) -- cycle;
    \draw (8,2) -- (6,2);
\end{tikzpicture} 
}%
\caption{Examples of eyes.}
\label{E}
\end{figure}
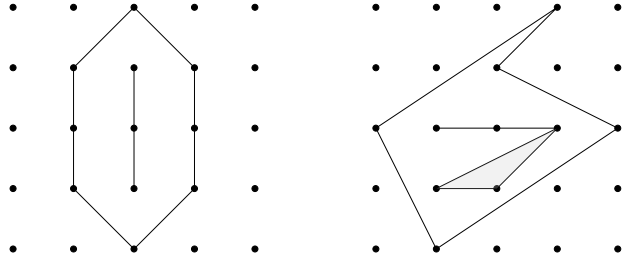

Eyes split the game board into subgames on potentially non-rectangular game boards.  Eyes cannot be claimed in a single move, so they are a structure that could easily pop up during game play.  However, sometimes small portions of eyes can still be claimed by a player before the entire eye is claimed. This can occur in two ways: using the vertices only from the iris or only from the boundary. In the former case, once no additional line segments can be added to connect the vertices of an iris, an iris is considered \textit{expanded}. If the iris is expanded and no additional line segments can be added to connect the boundary vertices we are left with a \textit{reduced eye}.  Note that this definition of a reduced eye, with an extra requirement of an expanded iris, preserves the intuitive idea that there is no immediately claimable area in a reduced shape. An \textit{extremely reduced eye} has the characteristics of a reduced eye but also the requirement that the eye cannot be subdivided in a single move.  Figure \ref{REye} shows the reduced eyes from Figure \ref{E}.  Close observation of these particular eyes shows they are not only reduced eyes but extremely reduced eyes, a property which turns out to be universal.

\begin{figure}[htb]
\centering
\resizebox{0.65\textwidth}{!}{%
\begin{tikzpicture}
    \foreach \x in {-1,0,...,3}{
      \foreach \y in {0,1,...,4}{
        \node[draw,circle,inner sep=1pt,fill] at (\x,\y) {};
      }
    }
    \draw (0,1) -- (1,0) -- (2,1) -- (2,3) -- (1,4) -- (0,3) -- cycle;
    \draw (1,1) -- (1,3);
    \draw[magenta, ultra thick] (0,2) -- (1,0) -- (0,3);
    \draw[magenta, ultra thick] (2,2) -- (1,0) -- (2,3);
    \fill[magenta!40, nearly transparent] (0,1) -- (1,0) -- (0,3) -- cycle;
    \fill[magenta!40, nearly transparent] (2,1) -- (1,0) -- (2,3) -- cycle;
    
    \foreach \x in {5,6,...,9}{
      \foreach \y in {0,1,...,4}{
        \node[draw,circle,inner sep=1pt,fill] at (\x,\y) {};
      }
    }
    \draw (6,0) -- (9,2) -- (7,3) -- (8,4) -- (5,2) -- cycle;
    \draw (6,1) -- (7,1) -- (8,2) -- cycle;
    \draw (8,2) -- (6,2);
    \fill[gray!40, nearly transparent] (6,1) -- (7,1) -- (8,2) -- cycle;
    \draw[magenta, ultra thick] (5,2) -- (7,3);
    \draw[magenta, ultra thick] (6,2) -- (6,1) -- (7,2);
    \fill[magenta!40, nearly transparent] (5,2) -- (7,3) -- (8,4) -- cycle;
    \fill[magenta!40, nearly transparent] (6,2) -- (6,1) -- (8,2) -- cycle;
\end{tikzpicture} 
}%
\caption{Claimed area that would result in reduced eyes.}
\label{REye}
\end{figure}
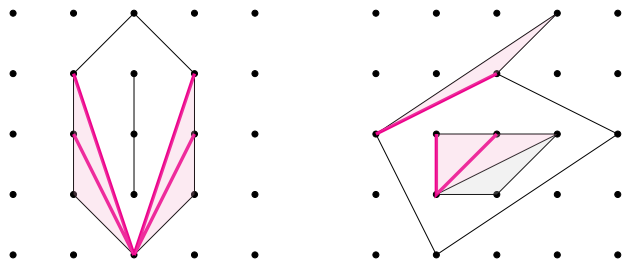

\begin{theorem}
All reduced eyes are extremely reduced eyes.
\label{ReducedExtremelyEyes}
\end{theorem}
\begin{proof}
Suppose there is a reduced eye that is not extremely reduced. Then there exists a valid move connecting two boundary points that divides the polygon into two shapes, both with interior points. But in a reduced eye, such a line must have passed through the iris, an invalid move. Thus by contradiction we've established the result.
\end{proof}
 
As a result of this theorem, we will often refer to extremely reduced eyes as just reduced eyes for simplicity. The next theorem shows us why players try to avoid being the first to move in a reduced eye. 

\begin{theorem}
The second player to move in a reduced eye gains its area.
\label{RedEyeArea}
\end{theorem}
 
\begin{proof}
Since a reduced eye is a closed shape with an expanded iris that is disconnected from the boundary points, the only valid move for a player to make is from a point on the boundary of that expanded iris to the boundary of the eye. The second player can then use that line segment, and the lattice structure of the game board, to form a claimed triangle. The second player can use at least one of the edges of this triangle to repeat the process until the entire eye is claimed. 
\end{proof}
 
While the eye structure is fairly straightforward in Dots-and-Triangles, it provides a greater challenge in the next game we discuss.

\section{Dots-and-Polygons}

Dots-and-Polygons is a generalized form of Dots-and-Triangles created by changing the second rule: If a player closes a polygon with no interior lines, they now claim that area and get an extra move. Note that closing two areas in the same move still only grants a single extra move.  By a polygon, we are referring to a closed shape without holes whose boundary is formed by line segments during game play which does not use a vertex or line segment multiple times in the formation of its boundary. See Figure \ref{Polygon}.

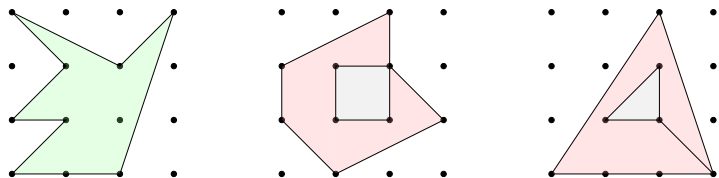
\begin{figure} [htb]
\centering
\resizebox{0.75\textwidth}{!}{%
\begin{tikzpicture}
    \foreach \x in {0,1,2,3}{
      \foreach \y in {0,1,2,3}{
        \node[draw,circle,inner sep=1pt,fill] at (\x,\y) {};
      }
    }
    \fill[green!40, nearly transparent] (0,0) -- (1,1) -- (0,1) -- (1,2) -- (0,3) -- (2,2) -- (3,3) -- (2,0) -- cycle;
    \draw (0,0) -- (1,1) -- (0,1) -- (1,2) -- (0,3) -- (2,2) -- (3,3) -- (2,0) -- cycle;

    \foreach \x in {5,6,7,8}{
      \foreach \y in {0,1,2,3}{
        \node[draw,circle,inner sep=1pt,fill] at (\x,\y) {};
      }
    }
    \fill[red!40, nearly transparent] (7,2) -- (7,3) -- (5,2) -- (5,1) -- (6,0) -- (6,2) -- cycle;
    \fill[red!40, nearly transparent] (6,0) -- (8,1) -- (7,2) -- (7,1) -- (6,1) -- cycle;
    \fill[gray!40, nearly transparent] (7,2) -- (6,2) -- (6,1) -- (7,1) -- cycle;
    \draw (7,2) -- (7,3) -- (5,2) -- (5,1) -- (6,0) -- (8,1) -- cycle;
    \draw (7,2) -- (6,2) -- (6,1) -- (7,1) -- cycle;
    
    \foreach \x in {10,11,12,13}{
      \foreach \y in {0,1,2,3}{
        \node[draw,circle,inner sep=1pt,fill] at (\x,\y) {};
      }
    }
    
    \fill[red!40, nearly transparent] (10,0) -- (13,0) -- (12,1) -- (11,1) -- cycle;
    \fill[red!40, nearly transparent] (10,0) -- (12,3) -- (13,0) -- (12,1) -- (12,2) -- (11,1) -- cycle;
    \fill[gray!40, nearly transparent] (12,1) -- (11,1) -- (12,2) -- cycle;
    \draw (10,0) -- (12,3) -- (13,0) -- cycle;
    \draw (13,0) -- (12,1) -- (11,1) -- (12,2) -- (12,1);

\end{tikzpicture} 
}%
\caption{The green area in the leftmost region fills a polygon, as defined, and thus is claimed area. The red areas in the other two figures are not claimed as neither fill a polygon.}
\label{Polygon}
\end{figure}

Using Pick's theorem, area can be easily calculated to find the winner of each game. Similar to Dots-and-Boxes and the previously discussed Dots-and-Triangles, a player performs double-dealing moves in order to gain control of the game. For example, consider the subgame in Figure \ref{DC}.  Robert can claim the entire triangular region's area of the left figure in his turn in several ways but is forced to move first in the next subgame. However, he can perform the double-dealing move in the right figure instead; Robert would gain no area but would force Betty to claim the area and move first in the next subgame.

\begin{figure} [htb]
\centering
\resizebox{0.75\textwidth}{!}{%
\begin{tikzpicture}
    \foreach \x in {0,1,2,3}{
      \foreach \y in {0,1,2}{
        \node[draw,circle,inner sep=1pt,fill] at (\x,\y) {};
      }
    }
    \foreach \x in {6,7,8,9}{
      \foreach \y in {0,1,2}{
        \node[draw,circle,inner sep=1pt,fill] at (\x,\y) {};
      }
    }
    \draw (0,0) -- (3,2) -- (0,1) -- (1,1);
    \draw (6,0) -- (9,2) -- (6,1) -- (7,1);
    \draw[red, thin, dashed] (6,1) -- (6,0);

\end{tikzpicture} 
}%
\caption{The figure on the left shows the original subgame.  The figure on the right shows $\R$'s double-dealing move.}
\label{DC}
\end{figure}
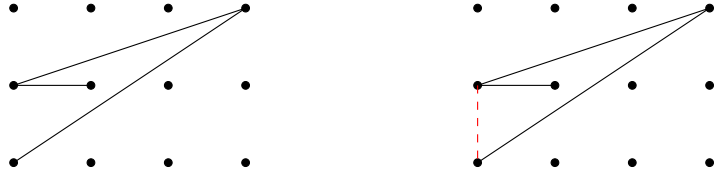

 The situation in Figure \ref{DC} can also occur in Dots-and-Triangles, but would not be an optimal double dealing move in that game (giving up $\frac{3}{2}$ area rather than 1 area).  Thus some doublecross situations in Dots-and-Triangles carry over to Dots-and-Polygons, but not all of them.  For instance, consider the previously discussed nested diamond double-dealing strategy from Figure \ref{NDPlay2}.  In this figure, a player was able to double-deal in the middle nested diamond by leaving several triangles unclaimed.  In Dots-and-Polygons, the player would have gained the area of these regions instead, since they are closed polygons with no line segments lying interior. From these two examples, we can observe that double-dealing moves are possible in Dots-and-Polygons, but it is a trickier feat to perform such moves, especially optimally.

\subsection{Claiming Area} Since Dots-and-Polygons allows larger swaths of area to be claimed, occasionally lattice points will never have any line segments connected to them. We adjust our definition of an eye slightly to account for this, with the modification that the interior points of the eye are either path connected or interior to a claimed region. See Figure \ref{HangingEyes}.

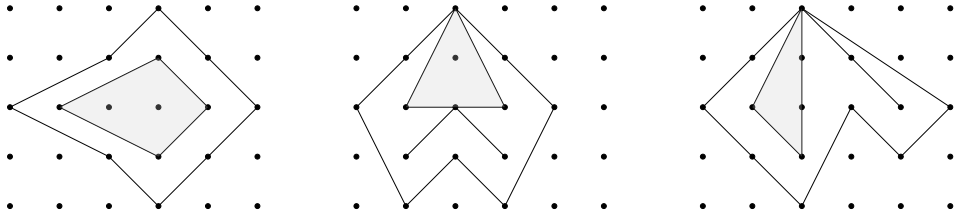
\begin{figure} [htb]
\centering
\resizebox{1\textwidth}{!}{%
\begin{tikzpicture}
    \foreach \x in {0,1,...,5}{
      \foreach \y in {0,1,...,4}{
        \node[draw,circle,inner sep=1pt,fill] at (\x,\y) {};
      }
    }
    \draw (3,1) -- (1,2) -- (3,3) -- (4,2) -- cycle;
    \fill[gray!40, nearly transparent] (3,1) -- (1,2) -- (3,3) -- (4,2) -- cycle;
    \draw (3,0) -- (2,1) -- (0,2) -- (2,3) -- (3,4) -- (5,2) -- cycle;

    \foreach \x in {7,8,...,12}{
      \foreach \y in {0,1,...,4}{
        \node[draw,circle,inner sep=1pt,fill] at (\x,\y) {};
      }
    }
    \draw (8,0) -- (7,2) -- (9,4) -- (8,2) -- (10,2) -- (9,4) -- (11,2) -- (10,0) -- (9,1) -- cycle;
    \draw (8,1) -- (9,2) -- (10,1);
    \fill[gray!40, nearly transparent] (9,4) -- (8,2) -- (10,2) -- cycle;
    
    \foreach \x in {14,15,...,19}{
      \foreach \y in {0,1,...,4}{
        \node[draw,circle,inner sep=1pt,fill] at (\x,\y) {};
      }
    }
    \draw (16,0) -- (14,2) -- (16,4) -- (15,2) -- (16,1) -- (16,4) -- (19,2) -- (18,1) -- (17,2) -- cycle;
    \draw (16,4) -- (18,2);
    \fill[gray!40, nearly transparent] (16,4) -- (15,2) -- (16,1) -- cycle;

\end{tikzpicture} 
}%
\caption{From left to right, the pictures show examples of an eye, a hanging eye, and a split hanging eye in Dots-and-Polygons.}
\label{HangingEyes}
\end{figure}

In Figure \ref{DC}, we can see that Robert performed a double-dealing move with the help of a ``hanging'' line segment. This motivates several of the following definitions. A \textit{hanging eye} has the same criteria as an eye with the relaxation that, instead of the interior vertices completely separated from the boundary vertices, the interior (the iris) is connected by one or more line segments to exactly one boundary vertex. Intuitively, a hanging eye resembles an eye whose iris has been connected to the boundary at a single vertex, as illustrated in Figure \ref{HangingEyes}. We also say that an eye is \textit{lazy} if the iris does not span the interior vertices of a closed region or include them interior to claimed area. If an eye has multiple irises spanning the interior vertices that are not path-connected to one another, we say it is \textit{split}. These terms can be used in conjunction with one another to describe shapes such as split hanging eyes as shown in Figure \ref{HangingEyes}.

While many of the theorems from Dots-and-Triangles will still hold in Dots-and-Polygons, the new nature of claiming area means each must be examined carefully. For instance, Theorem \ref{RedEyeArea} still holds for Dots-and-Polygons, but we have to carefully check that the move the first player makes does not form a valid polygon. 

\begin{theorem}
In Dots-and-Polygons, the second player to move in a reduced eye gains its area.
\end{theorem}

\begin{proof}
This proof follows similarly to Theorem \ref{RedEyeArea}; we need only check that the move the first player makes does not form a closed polygon. Such a move forms a line segment from a point on the boundary of that expanded iris to the boundary of the eye. This forms a closed region which uses the new line segment twice as a boundary. Thus this closed region is not a valid polygon, and the first player does not claim any area, ending their turn. This is illustrated in Figure \ref{InvalidArea}. The second player can claim the remaining area by simply making any move connecting a point on the boundary of the expanded iris to the boundary of the eye to form two valid closed polygons.
\end{proof}

\begin{figure}[htb]
\centering
\resizebox{0.65\textwidth}{!}{%
\begin{tikzpicture}
    \foreach \x in {0,1,...,4}{
      \foreach \y in {0,1,...,3}{
        \node[draw,circle,inner sep=1pt,fill] at (\x,\y) {};
      }
    }
    \draw (1,0) -- (4,2) -- (2,3) -- (0,2) -- cycle;
    \draw (1,1) -- (1,2) -- (3,2) -- (2,1) -- cycle;
    \fill[gray!40, nearly transparent] (1,1) -- (1,2) -- (3,2) -- (2,1) -- cycle;
    
    \foreach \x in {6,7,...,10}{
      \foreach \y in {0,1,...,3}{
        \node[draw,circle,inner sep=1pt,fill] at (\x,\y) {};
      }
    }
    \draw (7,0) -- (10,2) -- (8,3) -- (6,2) -- cycle;
    \draw (7,1) -- (7,2) -- (9,2) -- (8,1) -- cycle;
    \fill[gray!40, nearly transparent] (7,1) -- (7,2) -- (9,2) -- (8,1) -- cycle;
    \draw[magenta, ultra thick] (8,2) -- (8,3);
    \fill[red!40, nearly transparent] (8,3) -- (6,2) -- (7,0) -- (10,2) -- (8,3) -- (8,2) -- (9,2) -- (8,1) -- (7,1) -- (7,2) -- (8,2);

\end{tikzpicture} 
}%
\caption{The figure on the left shows a reduced eye.  The figure on the right illustrates a move from the boundary of the eye to the boundary of the iris.  The shaded red area marks a region that is not polygonal (and thus not claimable) as it uses one of its edges twice on the boundary.}
\label{InvalidArea}
\end{figure}
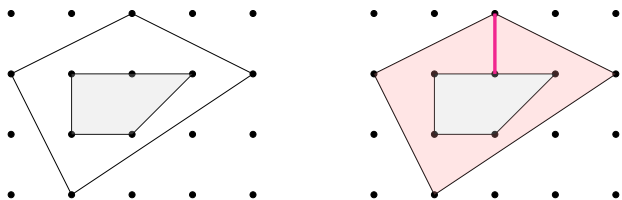

While claiming area in an expanded eye is straightforward, it is useful to note that the process of expanding an eye, and thus also an iris, is a ``free turn'' process; that is, if a player notes such a structure exists, they can quickly check and claim area in it until it is expanded. A player can expand the boundary of an eye by connecting pairwise combinations of boundary vertices with valid moves.  However, not every pairwise combination of iris vertices will form a valid polygon, and hence establishing that expanding an iris is a free turn process needs special consideration. This free turn process of expanding an eye will also imply that both hanging and split hanging eyes can be immediately claimed in a single turn. 

\begin{lemma}
The process of expanding an iris can always be done is such a way as to not consume a full turn.
\label{ExpandThm}
\end{lemma}

\begin{proof}
Start by connecting any vertices that result in claimed area. Such moves will always result in another free move. Continue this until no such move is possible. If the iris is expanded, we are done. If it is not expanded it means there are still two vertices that can be connected to close a region, but that region is not claimable. If the region is not claimable, then by the rules of the game, either the closed region is not a polygon or there are interior lines. We will show by contradiction that neither of these two cases can occur.

\textit{Case 1:} First consider the case where the closed region is not a polygon. That is, suppose we connect two vertices to form a unclaimable closed region without interior lines. Note that such a region must use a single vertex $i$ twice on its boundary. (A closed region that uses line segments twice on its boundary can be classified as case 2). We will show there is a region that should have already been claimed. In the scenario where we use $i$ twice, consider the vertices near $i$ and on the opposite ``side'' of the the line segment that was added to close the region (see Figure \ref{SharedVertex}).  Since we are on a lattice structure and there are no interior vertices or lines in the region we are guaranteed that there exist two such vertices that can be connected validly. Connecting these two then forms a polygon with no interior points, a claimable region that leads to a contradiction as any such area should have already been claimed.

\begin{figure}[htb]
\centering
\resizebox{0.5\textwidth}{!}{%
\begin{tikzpicture}
    \foreach \Point in {(0,3),(0,5),(2,2)[label=$i$],(3.5,3),(0,0)}{
        \node at \Point {\textbullet};
    }

    \draw[thick] (0, 3) -- (0,5) -- (2,2) -- (0,0);
    \draw[magenta,thick] (2,2) -- (3.5,3);
    \draw[dash dot] (3.5,3) -- (5,4);
    \draw[dash dot] (2,2) -- (4,1);
    \draw[dash dot] plot [smooth,tension=0.5] coordinates{(0,0) (-0.4,-0.4) (-0.6,-0.6) (-0.5,-1) (2,-1.4) (3.2,-0.2) (3.6,0.5) (4,1)};
    \draw[dash dot] plot [smooth,tension=0.5] coordinates{(0,3) (-0.5,1.5) (-0.8,0.4) (-1.2,0) (-1,-1) (0.2,-1.7) (1.5,-2) (2.4,-2.4) (3, -1.3) (3.6,-1) (3.8,-0.3) (4,0.2) (4.4,1) (4.6, 2) (5.7,3) (5,4)};
    \fill[gray,nearly transparent]  plot [smooth,tension=0.5] coordinates{(0,0) (-0.4,-0.4) (-0.6,-0.6) (-0.5,-1) (2,-1.4) (3.2,-0.2) (3.6,0.5) (4,1)};;
    \fill[gray,nearly transparent] (0,0) -- (2,2) -- (4,1) -- cycle;
    \draw[pattern color=yellow, pattern=north east lines] (0,0) -- (0,5) -- (2,2) -- cycle;
    \draw[yellow,thick] (0,0) -- (0,3);
\end{tikzpicture} 
}
\caption{A region that is not a polygon, closed by adding the pink line (not necessarily adjacent to $i$) on the right side. We are considering the vertices on the left. The shaded yellow region indicates area that should have already been claimed by placing the yellow line segment.}
\label{SharedVertex}
\end{figure}
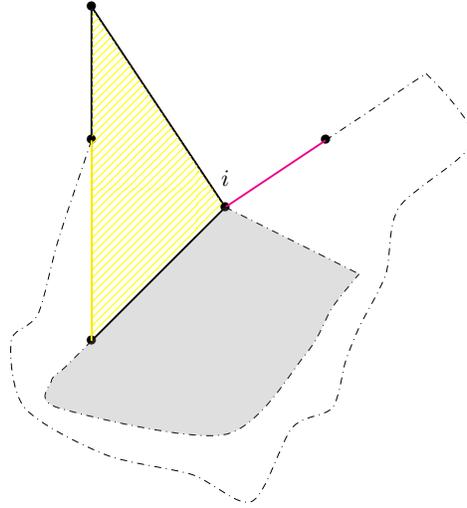

\textit{Case 2:} Now consider the case where the closed region has interior lines.  Let $e_1$ be the line segment that was added to close this region. By the definition of an iris, all lines must be path connected (potentially using the boundary of this new closed region), and it contains no degree 0 vertices. This means this new closed portion of the iris is itself either a hanging eye or split hanging eye that spans the interior vertices.

If we have a hanging eye, draw another line $e_2$ between the iris and the border of the closed region. If we have a split hanging eye draw $e_2$ between two hanging irises. This splits the region into two parts, one with both $e_1$ and $e_2$ on its boundary, and one with just $e_2$. This new region formed with $e_2$ is not claimable, as otherwise it would have been claimed by supposition. Thus, the new region is a (smaller) region to which we can again apply either case 1 or case 2. Since case 1 has been disproved, we can repeat the process from case 2. Since we are working on a finite lattice structure, we know that there are only a finite number of possible moves available to us in this region. Eventually, drawing a line segment $e_k$ would form a triangle or other claimable polygon that did not contain any other $e_j$, for $1\leq j<k$.  However, this contradicts the fact we started by claiming all possible area. Hence, this case also does not hold.
\end{proof}

\begin{theorem}
A hanging eye can have its entire area claimed in a single turn.
\label{HangEyeOneTurn}
\end{theorem}

\begin{proof}
We begin by applying Lemma \ref{ExpandThm} to the iris of the hanging eye.  The process of expanding the iris does not consume a full turn nor claim all the area of the eye.  Expanding the iris leaves us with an expanded hanging eye. Due to the lattice structure and no degree 0 interior points, we can find a move connecting the expanded iris to the boundary. This divides the hanging eye into two valid polygons and thus claims the area.

\end{proof}

Similarly, we can argue a split hanging eye can also have its entire area claimed in a single turn. Theorem \ref{HangEyeOneTurn} implies that a hanging eye cannot be reduced by definition.  Additionally, even though a player can claim the entire area of a hanging or split hanging eye, they may want to only claim part of the area in order to create a doublecrossed move. This theorem cannot, however, be extended to a lazy hanging or lazy split hanging eye where the iris(es) do not span all the interior nodes. While a player may still claim the region in a single turn, this will not always hold, as pictured in Figure \ref{HangerCE}.

\begin{figure} [htb]
\centering
\resizebox{0.35\textwidth}{!}{%
\begin{tikzpicture}
    \foreach \x in {0,1,...,4}{
      \foreach \y in {0,1,...,4}{
        \node[draw,circle,inner sep=1pt,fill] at (\x,\y) {};
      }
    }
    \draw (2,0) -- (4,3) -- (2,4) -- (0,3) -- cycle;
    \draw (2,4) -- (2,3);

\end{tikzpicture} 
}%
\caption{Here the figure is a lazy hanging eye where the iris does not span all vertices.  A player is unable to claim this region in a single turn.}
\label{HangerCE}
\end{figure}
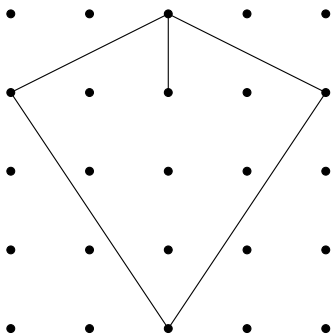

\subsection{Double-Dealings and Doublecrosses}

Now, in Dots-and-Triangles Theorem \ref{TrianglesDCThm} established that the number of turns equals the number of dots on the game board plus the number of doublecrosses. However, this is only an upper bound on the number of moves in Dots-and-Polygons. That is, claiming any closed polygon in Dots-and-Polygons takes at most the same number of moves as triangulating that same polygon in Dots-and-Triangles. However, in the next theorem we establish an equality. 

\begin{theorem}
The number of turns equals the number of dots on the game board plus the number of double crosses minus the number of unused interior dots
\end{theorem}

\begin{proof}

First, suppose we are playing a game of Dots-and-Polygons without any doublecrosses and every claimed polygon has no unused interior points. Let $D$ be the number of dots, $T$ the number of turns to draw $L$ line segments, and $P$ be the number of polygons we finish the game with. Now, every line segment that is placed, except the last, either forms exactly one polygon or ends the turn (but not both). Then the similar result from the Dots-and-Triangles proof holds, and we get that the number of turns equals the number of dots.  If we allow doublecrosses still without unused interior points, the result still follows that we have the number of turns equals the number of dots plus the number of doublecrosses.

Now, consider what happens if there are unused interior points.  Let $I$ stand for the number of these unused interior points, and $C$ be the number of doublecrosses.  So, it still follows that
$$L=P-C+T-1.$$
Using Euler's formula where $D-I$ is the number of vertices, $L$ is the number of edges, and $P+1$ is the number of faces, we get that
$$(D-I)-L+(P+1)=2.$$
Employing some algebra gives us the result
\[
T=D+C-I.\qedhere
\]
\end{proof}

We can assume in Dots-and-Polygons that any doublecrossed move must occur within a closed region. Imagine Robert performs a double-dealing move, but leaves an open region after his turn.  Betty will perform a doublecrossed move on her turn, and as a result, will claim the entire region during her turn. If Betty closes the region, then Robert's move was not double-dealing as he could not have claimed the entire area to begin with, although he may have been able to claim part of it.  Alternatively, if Betty does not close the region, then she did not claim any area, and hence did not make a doublecrossed move. Therefore, we can assume that all further doublecrossed moves occur inside a closed region. Similarly, we can use this reasoning to establish the following theorem.

\begin{theorem}
A player is unable to perform a double-dealing move in a claimable region with no interior points.
\label{NoDC}
\end{theorem}

\begin{proof}
Suppose that Robert can perform a double-dealing move against Betty in a claimable region with no interior points.  Then the polygon is missing exactly one line segment on its boundary; otherwise it is not claimable in a single turn. Robert must still perform a double-dealing move by supposition.  However in this region such a move would complete a valid polygon and claim area, resulting in an additional turn, which is a contradiction to the double-dealing strategy.
\end{proof}

\begin{theorem}
A player can only be forced to perform a doublecrossed move in a closed region with a hanging eye or split hanging eyes.
\label{DCHangEye}
\end{theorem}

\begin{proof}
We will assume that Robert performed the double-dealing move, and Betty is the player forced into making the doublecrossed move. Proceeding via contradiction, assume that Betty performs the doublecrossed move in a region that does not meet the claimed criteria.  This yields two cases.

\textit{Case 1:} Suppose not all the interior vertices are spanned. Then Betty can simply connect a degree 0 vertex to any other vertex without claiming area as it is impossible to have a polygon with a degree 1 vertex.
 
\textit{Case 2:} Suppose all the interior vertices are spanned but there is at least one iris that is not hanging. Note that the line segments that form the non-hanging portion of the irises are disjoint from all other line segments that form the region. This means if Betty connects a vertex from the non-hanging iris to any other vertex, she would not claim area as the region formed would use the new line segment as a boundary twice (and thus would not be a polygon).

Both these cases contradict Betty being forced to perform the doublecrossed move.
\end{proof}

The previous theorem gives necessary conditions for a doublecrossed move. Note that this theorem describes not the region in which we doublecross, but the region we get as a result of the doublecross. However, this theorem does not stipulate that every closed region with hanging eyes can be used to force a player into a doublecrossed move, as is demonstrated in Figure \ref{UnableToDC}.

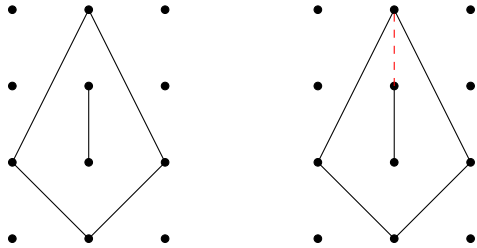
\begin{figure} [htb]
\centering
\resizebox{0.5\textwidth}{!}{%
\begin{tikzpicture}
    \foreach \x in {0,1,2}{
      \foreach \y in {0,1,...,3}{
        \node[draw,circle,inner sep=1pt,fill] at (\x,\y) {};
      }
    }
    \draw (1,0) -- (0,1) -- (1,3) -- (2,1) -- cycle;
    \draw (1,1) -- (1,2);
    
    \foreach \x in {4,5,6}{
      \foreach \y in {0,1,...,3}{
        \node[draw,circle,inner sep=1pt,fill] at (\x,\y) {};
      }
    }
    \draw (5,0) -- (4,1) -- (5,3) -- (6,1) -- cycle;
    \draw (5,1) -- (5,2);
    \draw[red, thin, dashed] (5,2) -- (5,3);

\end{tikzpicture} 
}%
\caption{The figure on the left shows the initial region $\R$ can move in. Although the move shown in the figure on the right creates a hanging eye, it does not qualify as a double-dealing move as $\R$ could not have claimed area initially.}
\label{UnableToDC}
\end{figure}

We have examined regions in which players can make double-dealing moves, but what is the smallest area a player would have to give up when making such a move?  Since the goal of a double-dealing move is to give up area to gain more area in the long run, players might try to strategize their double-dealing moves to give up as little area as possible.

\begin{theorem}
The minimum amount of area that is needed to perform a double-dealing move is $\frac{3}{2}$.
\end{theorem}

\begin{proof}
Since we are working on an integer lattice, then the only regions with area less than $\frac{3}{2}$ we need to consider are $\frac{1}{2}$ and $1$. By Pick's theorem both areas would not have any interior points. Suppose a player makes a double-dealing move in such regions. Since there are no interior points this means the region is either not closed by the double-dealing move or it is a closed region that is immediately claimed. The former case is a contradiction to Theorem \ref{UnableToDC} while the latter case precludes a double dealing move, indicating the original move must not have been double-dealing. Hence we rule out regions with area $\frac{1}{2}$ or $1$.

We will now show that a double-dealing move exists for a shape with an area of $\frac{3}{2}$. We have two cases to consider by Pick's Theorem: $I=0$ and $B=5$, or $I=1$ and $B=3$. In the former case, Theorem \ref{NoDC} tell us that a doublecross cannot occur. This means that we have a triangular region with one interior point.  Consider the region in Figure \ref{DC}.  In this, we can see that a player can introduce the final boundary line to complete the double-dealing move.
\end{proof}

The previous theorem gives players a goal to strive for when performing double-dealing moves. Unfortunately, an explicit strategy to ensure optimal double-dealing moves with minimum amount of ceded area has yet to be established. Moreover, guaranteeing the existence of a double-dealing move is also tricky.  If we consider Figure \ref{NoDCPicture} we have a split hanging eye, but upon closer inspection, it is evident there is no way to do a double-dealing move.

\begin{figure} [htb]
\centering
\resizebox{0.3\textwidth}{!}{%
\begin{tikzpicture}
    \foreach \x in {0,1,...,4}{
      \foreach \y in {0,1,...,3}{
        \node[draw,circle,inner sep=1pt,fill] at (\x,\y) {};
      }
    }
    \draw (0,1) -- (1,0) -- (2,0) -- (3,1) -- (4,3) -- (1,2) -- cycle;
    \draw (2,0) -- (1,1) -- (2,1) -- (2,2);
    \draw (2,0) -- (3,2);

\end{tikzpicture} 
}%
\caption{A split hanging eye that does not contain a double-dealing move.}
\label{NoDCPicture}
\end{figure}
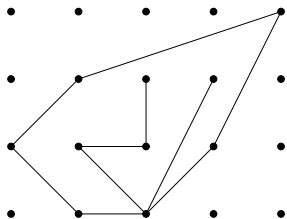

\section{Conclusion}
We have analyzed game play strategies in two proposed variations of the children's game Dots-and-Boxes, Dots-and-Triangles and Dots-and-Polygons, including several winning strategies similar to those in \cite{Berlekamp,Win}. However, there are still many future directions to explore. What happens if we play these games on various surfaces, such as a torus? What strategies guarantee a player is giving up the minimum amount of area after performing a double-dealing move? Additionally, Theorem \ref{ConvexERS} showed there are no convex extremely reduced shapes without interior line segments that have five boundary points. Can this be extended to the cases with larger odd numbers of boundary points?

\bibliographystyle{amsplain}
\bibliography{Dots_and_Polygons}

\end{document}